\documentclass{birkau}
\usepackage{amsfonts}
\usepackage{amssymb}
\usepackage{amsmath}
\usepackage{amsrefs}
\usepackage{tikz-cd}
\usepackage{centernot}
\usepackage{enumerate}
\usepackage{fullpage}
\usepackage{multicol}
\usepackage{amsthm,amsmath,amssymb,tikz,tikz-cd,amsmath,mathrsfs,mathtools,multicol,dirtytalk,tabularx,xy,lipsum,url,enumitem,cmll}
\usepackage{float}

\newtheorem{thm}{Theorem}[section]
\newtheorem{lem}[thm]{Lemma}
\newtheorem{cor}[thm]{Corollary}
\newtheorem{prop}[thm]{Proposition}
\newtheorem{rem}[thm]{Remark}

\theoremstyle{definition}
\newtheorem{defn}[thm]{Definition}

\theoremstyle{remark}

\newtheorem{ex}[thm]{Example}

\newcommand{\mc}{\mathcal}
\newcommand{\mf}{\mc}

\newcommand{\ol}{\overline}
\newcommand{\ts}{\textsc}

\newcommand{\Ex}{\exists}

\allowdisplaybreaks

\begin{document}

\title{Monadic ortholattices: completions and duality}

\corrauthor{John Harding}
\address{Department of Mathematical Sciences\\New Mexico State University\\ Las Cruces 88003\\USA}
\email{jharding@nmsu.edu}

\author{Joseph McDonald}
\address{Department of Mathematical Sciences\\New Mexico State University\\ Las Cruces 88003\\USA}
\email{jsmcdon1@ualberta.ca}

\author{Miguel Peinado}
\address{Department of Mathematical Sciences\\New Mexico State University\\ Las Cruces 88003\\USA}
\email{peinadop@nmsu.edu}

\thanks{The first and third listed authors were partially supported by US Army grant W911NF-21-1-0247 and the first author was also partially supported by NSF grant DMS-2231414. The second author was supported by CGS-D MSFSS grant no. 771-2023-0044 and CGS-D SSHRC grant no. 767-2022-1514.}

\subjclass{06C15, 06B23 06E15.}

\keywords{Monadic ortholattice, MacNeille completion, canonical completion, duality.}
\begin{abstract} 
We show that the variety of monadic ortholattices is closed under MacNeille and canonical completions. In each case, the completion of $L$ is obtained by forming an associated dual space $X$ that is a monadic orthoframe. This is a set with an orthogonality relation and an additional binary relation satisfying certain conditions. For the MacNeille completion, $X$ is formed from the  non-zero  elements of $L$, and for the canonical completion, $X$ is formed from the proper filters of $L$. The corresponding completion of $L$ is then obtained as the ortholattice of bi-orthogonally closed subsets of $X$ with an additional operation defined through the binary relation of $X$. 

With the introduction of a suitable topology on an orthoframe, as was done by Goldblatt and Bimb\'o, we obtain a dual adjunction between the categories of monadic ortholattices and monadic orthospaces. A restriction of this dual adjunction provides a dual equivalence. 
\end{abstract}

\maketitle

\section{Introduction}
Monadic algebras were introduced by Halmos \cite{halmos} as an algebraic realization of the one-variable fragment of first-order logic. A monadic algebra is a Boolean algebra with an additional unary operation $\Ex$, called a quantifier, whose closed elements are a Boolean subalgebra. Halmos' polyadic algebras \cite{halmos} had a family of interrelated quantifiers and played the same role for full first-order logic. At about the same time, Henkin, Monk, and Tarski \cites{tarski1,tarski2} introduced the closely related cylindric algebras as algebraic models of first-order logic. These too had a family of quantifiers, related in a somewhat different way than in polyadic algebras. They showed that each monadic algebra and each cylindric algebra can be embedded into a complete, atomic one. These were among the results that grew into the theory of canonical extensions of Boolean algebras with operators \cites{jonsson,jonssonII}. The approach was to show that the quantifier of a monadic algebra yields an equivalence relation on its set of ultrafilters, and the powerset of this relational structure is then a complete atomic monadic algebra extending the original. Including the Stone topology into this process yields a duality between monadic algebras and Stone spaces equipped with a compatible equivalence relation. 

An \emph{ortholattice} (abbrev.: \ts{ol}) is a bounded lattice with an order-inverting period two complementation. A \emph{monadic} \ts{ol} is an ortholattice with a quantifier, a closure operation whose closed elements are a sub-\ts{ol}. Janowitz \cite{janowitz} first considered quantifiers on orthomodular lattices, and Harding \cite{harding} studied them, and cylindric \ts{ol}s, for their connections to von Neumann algebras, in particular, to subfactors. The broad purpose of this note is to conduct a study for monadic \ts{ol}s similar to that described for monadic algebras. We use a number of tools for this purpose. 

An \emph{orthoframe} (abbrev.: \ts{of}) is a set $X$ with an \emph{orthogonality relation}, a binary relation $\perp$ that is irreflexive and symmetric. Orthogonality relations are special examples of the polarities described by Birkhoff \cite{BLT}. It is known that the bi-orthogonally closed sets of an \ts{of} form a complete \ts{ol}. There are two well-used ways to construct an orthoframe from an \ts{ol}: with $X$ the set of non-zero elements of $L$, which we call \emph{MacLaren's} \ts{of} (see \cite{maclaren}), and with $X$ the set of proper filters of $L$, which we call \emph{Goldblatt's} \ts{of} (see \cite{goldblatt}).

In \cite{goldblatt}, Goldblatt introduced a topology on what we call the Goldblatt frame of an \ts{ol} $L$. This has all sets $h(a)=\{x:a\in x\}$ for $a\in L$, and their set-theoretic complements, as a sub-basis. Goldblatt showed that this yields a Stone topology, and that the clopen bi-orthogonally closed sets of the Goldblatt frame form an \ts{ol} that is isomorphic to $L$. Bimb\'o \cite{bimbo} introduced \emph{orthospaces} (abbrev.: \ts{os})
as certain \ts{of}s with a Stone topology and order. She defined morphisms between \ts{os}s and thought to have produced a duality between the category of ortholattices and their homomorphisms and the category of orthospaces and their morphisms. We show that what is produced in \cite{bimbo} a dual adjunction and with an additional condition on \ts{os}s called \emph{ortho-sobriety}, a dual equivalence is obtained \cite{dmitrieva}. 

Harding \cite{harding} defined \emph{monadic orthoframes} to be \ts{of}s with an additional binary relation satisfying certain conditions. He showed that the bi-orthogonally closed elements of a monadic \ts{of} form a monadic \ts{ol}, and that the MacLaren \ts{of} of a monadic \ts{ol} can be turned into a monadic \ts{of} whose bi-orthogonally closed elements contain the original monadic \ts{ol} as a subalgebra. 

In the second section of this note we provide preliminaries. In the third section we show that for a monadic \ts{ol} $L$, the bi-orthogonally closed sets of the monadic \ts{of} constructed in \cite{harding} give the MacNeille completion of $L$ in the sense of \cite{Gehrke}. Thus, the variety of monadic \ts{ol}s is closed under MacNeille completions, hence by \cite{Gehrke} it is also closed under canonical completions. We next provide a similar description of the canonical extension of $L$ via monadic \ts{of}s. As a first step, we show that the bi-orthogonally closed sets of the Goldblatt \ts{of} of an \ts{ol} is its canonical completion. Then we construct from a monadic \ts{ol} $L$ a monadic \ts{of} structure on its Goldblatt \ts{of} and show that the bi-orthogonally closed subsets of this monadic \ts{of} yield the canonical extension of $L$.

In the fourth section we consider orthospaces. We provide an example to show that there is not a dual equivalence between \ts{ol}'s and \ts{os}'s, and use the remainder of the section to show that there is a dual adjunction between these categories. This dual adjunction restricts to a dual equivalence with the additional condition of \emph{ortho-sobriety} on an \ts{os} as was pointed out by Dmitrieva in \cite{dmitrieva}. 

In the final section, we adapt this adjunction to the setting of monadic orthoframes and call the resulting structures \emph{monadic orthospaces}. We then show that there is a dual adjunction between the categories of monadic \ts{ol}s and monadic \ts{os}'s and that this provides a dual equivalence when restricted to the full sub-category of monadic \ts{os}s consisting of ortho-sober monadic \ts{os}s.

\section{Preliminaries}

\begin{defn}
An  \emph{ortholattice} $(L,\wedge,\vee,',0,1)$ is a bounded lattice with an order-inverting period two complementation. A \emph{monadic} \ts{ol} is an \ts{ol} with a quantifier $\Ex$, i.e. a closure operator where the orthocomplement of a closed element is closed.
\end{defn}

For an \ts{ol} $(L,\wedge,\vee,',0,1)$ we use $L$ to denote both the \ts{ol} and its underlying set since this will not cause confusion. We let $L^*$ be the set of non-zero elements of $L$ and $\mf{F}(L)$ be the set of proper, non-empty filters of $L$ ordered by set inclusion. We use letters such as $a,b,c,etc.$ for elements of $L$ and $x,y,z,etc.$ for elements of $\mf{F}(L)$.
\begin{defn}
Let $L$ be an \ts{ol}. For $a,b\in L^*$ set $a\perp b$ iff $a\leq b'$, and for $x,y\in\mf{F}(L)$ set $x\perp y$ iff there is $a\in L^*$ with $a\in x$ and $a'\in y$. 
\end{defn}

It is obvious that both relations are irreflexive and symmetric. 

\begin{defn}
Call $(L^*,\perp)$ the \emph{MacLaren} \ts{of} of $L$ and $(\mf{F}(L),\perp)$ the \emph{Goldblatt} \ts{of} of $L$. 
\end{defn} 

For an \ts{of} $(X,\perp)$ we use $X$ to denote both the \ts{of} and its underlying set since this will not cause confusion. For $S\subseteq X$ its \emph{orthogonal} is $S^\perp=\{y\in X:x\perp y\mbox{ for all }x\in S\}$, and its \emph{bi-orthogonal} is $S^{\perp\perp}$. Call $S$ \emph{bi-orthogonally closed} if $S=S^{\perp\perp}$. 

\begin{defn}
Let $\mc{B}(X)$ be the set of bi-orthogonally closed subsets of an \ts{of} $X$.    
\end{defn}

It is well-known that this is a complete \ts{ol} with partial ordering of set inclusion and with the orthocomplement of $S$ given by $S^\perp$. In this \ts{ol} meets are given by intersections, joins by the bi-orthogonal of the union, and the bounds are the emptyset and $X$. 

\begin{prop}\label{stuff}
Suppose $L$ is an \ts{ol}. Then there is an \ts{ol} embedding $g:L\to \mc{B}(L^*,\perp)$ with $g(a)=\{b\in L^*:b\leq a\}$, and an \ts{ol}-embedding $h:L\to\mc{B}(\mf{F}(L),\perp)$ with $h(a)=\{x:a\in x\}$. 
\end{prop}

It is known \cite{maclaren} that $g:L\to\mc{B}(L^*,\perp)$ is the MacNeille completion of $L$. In the next section, we show that $h:L\to\mc{B}(\mf{F}(L),\perp)$ is the canonical extension of $L$ \cite{harding2}. In the following, for a binary relation $R$ on a set $X$ and $A\subseteq X$, we denote the relational image of $A$ by $R[A]=\{y\in X:xRy$ for some $x\in A\}$. 

\begin{defn}\label{monadic orthoframe}
A \emph{monadic orthoframe} is a triple $(X,\perp,R)$ where $(X,\perp)$ is an \ts{of} and $R$ is a reflexive, transitive binary relation on $X$ that satisfies $R[R[\{x\}]^{\perp}]
\subseteq R[\{x\}]^{\perp}$ for all $x\in X$. 
\end{defn}

The following results were established by Harding in \cite{harding}.

\begin{prop}
For $X$ a monadic \ts{of}, its bi-orthogonally closed subsets $\mc{B}(X)$ form a monadic \ts{ol} under the quantifier $\Ex A = R[A]^{\perp\perp}$. 
\end{prop}

\begin{prop}
For $L$ a monadic \ts{ol}, the relation $R$ on $L^*$ defined by $aRb$ iff $b\leq \Ex a$ makes $(L^*,\perp,R)$ a monadic \ts{of}, and the map $g:L\to\mc{B}(L^*,\perp,R)$ is a monadic \ts{ol} embedding.  
\end{prop}


\section{MacNeille and canonical completions}

For $L$ a bounded lattice, an $n$-ary operation $f:L^n\to L$ is called \emph{monotone} if in each coordinate it either preserves or reverses order. Implication of a Heyting algebra is monotone, and both the quantifier and orthocomplementation of a monadic \ts{ol} are monotone. A lattice with additional operations is monotone if each of its operations is monotone, and a variety of lattices with additional operations is monotone if each of its members is monotone. There is a theory of completions of lattices with monotone operations that we describe in the restricted case of the variety of monadic \ts{ol}s. 

For a bounded lattice $L$, its \emph{MacNeille completion} is a pair $(e,\ol{L})$ where $\ol{L}$ is a complete lattice, $e:L\to\ol{L}$ is a lattice embedding, and each element of $\ol{L}$ is both a join and a meet of elements of the image of $L$. For a monadic \ts{ol} $L$, its MacNeille completion is the bounded lattice $\ol{L}$ with unary operations $\ol{'}$ and $\ol{\Ex}$ defined by 
\begin{align*}
x^{\ol{'}}&=\bigwedge\{e(a'):e(a)\leq x\}\\
\ol{\Ex}x &= \bigvee\{e(\Ex a):e(a)\leq x\}
\end{align*}

\begin{prop}
For $L$ a monadic \ts{ol}, $g:L\to\mc{B}(L^*,\perp,R)$ is its MacNeille completion. 
\end{prop}

\begin{proof}
On the \ts{ol} level this is well known \cite{maclaren}. It remains to show that for $A\subseteq L^*$ bi-orthogonally closed, i.e. for a normal ideal $A$ of $L^*$, we have\[R[A]^{\perp\perp}=\bigvee\{R[g(a)]^{\perp\perp}:a\in A\}.\]
Since the join of bi-orthogonally closed sets is given by the closure of their union, the right side of this expression is equal to $(\bigcup\{R[g(a)]^{\perp\perp}:a\in A\})^{\perp\perp}$. By general principles, this in turn is equal to $(\bigcup\{R[g(a)]:a\in A\})^{\perp\perp}$, hence to $R[\bigcup\{g(a):a\in A\}]^{\perp\perp}$. But $g(a)$ is the principle ideal generated by $a$ and $A$ is a normal ideal, so $\bigcup\{g(a):a\in A\}=A$. 
\end{proof}

\begin{defn}
For a bounded lattice $L$, its \emph{canonical completion} is a pair $(e,C)$ where $C$  is a complete lattice and $e:L\to C$ is a bounded lattice embedding that is dense and compact. \emph{Dense} means that each element of $C$ is both a join of meets and a meet of joins of elements of the image of $L$. \emph{Compact} means that if $S,T\subseteq L$ then 
\[\bigwedge e[S]\leq\bigvee e[T]\quad \Rightarrow\quad \bigwedge e[S']\leq\bigvee e[T']\] 
for some finite $S'\subseteq S$ and $T'\subseteq T$.
\end{defn}

Each lattice has up to isomorphism a unique canonical completion, and we call this \emph{the} canonical completion, denoted by $L^\sigma$. An element of the canonical completion that is a meet of elements of the image of $L$ is called closed, and the set of closed elements is $\mc{K}$. For a bounded lattice with additional monotone operations, there are extensions of the operations to the canonical completion. We describe these for orthocomplementation and a quantifier, where we call the extensions $'^\sigma$ and $\Ex^\sigma$, by
\begin{align*}
x^{'^\sigma}&\,\,=\,\,\bigwedge\{\bigvee\{ e(a'): k\leq e(a) \} : k\leq x \mbox{ and } k\in \mc{K} \},        \\
\Ex^\sigma x &\,\,=\,\, \bigvee\{\bigwedge\{ e(\Ex a):k\leq e(a) \} : k\leq x \mbox{ and } k\in \mc{K} \}.        
\end{align*}

\begin{prop}
For $L$ an \ts{ol}, $h:L\to\mc{B}(\mf{F}(L),\perp)$ is its canonical extension. 
\end{prop}

\begin{proof}
For convenience, we use $\mc{B}$ for $\mc{B}(\mf{F}(L),\perp)$. Recall that for any $a\in L$, that $h(a)$ is the set of all proper filters of $L$ that contain $a$. 
By Proposition~\ref{stuff}, $h:L\to \mc{B}$  is an \ts{ol}-embedding and $\mc{B}$ is complete. Meets in $\mc{B}$ are given by intersections, joins are the bi-orthogonal of the union. In particular, $(h,\mc{B})$ is a completion. For $S\subseteq L$, note that \[\bigwedge h[S] = \bigcap h[S] = \{x:S\subseteq x\}.\] 

To show that this completion is dense, suppose that $A\in\mc{B}$. From the comments above, we have that $A\subseteq \bigcup\{\bigcap h[x]:x\in A\}$. To see equality, suppose $y\in \bigcap h[x]$ for some $x\in A$. Then $x\subseteq y$. Since $A \in\mc{B}$, it is bi-orthogonally closed, and it follows that $A$ is an upset in the filter lattice. Thus, since $x\in A$ we have $y\in A$. So $A=\bigcup\{\bigcap h[x]:x\in A\}$. Since $A$ is bi-orthogonally closed, so is this union, and it follows that $A=\bigvee\{\bigwedge h[x]:x\in A\}$. Thus each element of $\mc{B}$ is a join of meets of elements in the image of $L$. Since $\mc{B}$ is an \ts{ol}, and De~Morgan's laws extend to arbitrary joins and meets in an \ts{ol}, it follows that each element of $\mc{B}$ is also the meet of joins of elements from the image of $L$. Thus this completion is dense.


To see that this completion is compact, suppose that $S,T\subseteq L$ and $\bigwedge h[S] \subseteq \bigvee h[T]$. We must show there are finite $S'\subseteq S$ and $T'\subseteq T$ with $\bigwedge h[S'] \subseteq\bigvee h[T']$. Let $x$ be the filter generated by $S$ and let $I$ be the ideal of $L$ generated by $T$. If $0\in x$ or $1\in I$, the result is trivial, so we assume that both are proper. Set $z=\{a':a\in I\}$ and note that $z$ is a proper filter. It is the filter generated by $\{a':a\in T\}$. Since $S\subseteq x$ we have $x\in\bigwedge h[S]$, hence $x\in\bigvee h[T]$. Note that since $\bigvee h[T]$ is closed, it is equal to $(\bigvee h[T])^{\perp\perp}$. Using this, the generalized De~Morgan's laws, and the fact that $h(a)^\perp=h(a')$, we have 
\[(\bigvee h[T])^{\perp\perp} = (\bigvee\{h(a):a\in T\})^{\perp\perp} = (\bigwedge\{h(a'):a\in T\})^\perp.\]
Since $a'\in z$ for each $a\in T$, we have $z\in\bigwedge\{h(a'):a\in T\}$. Since $x\in \bigvee h[T]$, we then have that $x\perp z$. So there is some $b\in L$ with $b\in x$ and $b'\in z$. Recall that $x$ is the filter generated by $S$. Also, since $z$ is the filter generated by $\{a':a\in T\}$, it follows that $b$ is in the ideal generated by $T$. It follows that there are finite subsets $S'\subseteq S$ and $T'\subseteq T$ with $\bigwedge S'\leq  b\leq\bigvee T'$, hence $\bigwedge h[S']\subseteq\bigvee h[T']$.  

We have shown that $h:L\to\mc{B}$ is the canonical completion of the lattice $L$. It remains to see that the orthocomplementation of $\mc{B}$ is the extension $'^\sigma$ described above. This is a consequence of the general De~Morgan laws and density of the canonical extension. 
\end{proof}

\begin{rem} {\em 
It appears that Goldblatt \cite[p.~47]{goldblatt} claims that if $L$ is a complete \ts{ol}, then $h$ maps $L$ isomorphically onto $\mc{B}(\mf{F}(L),\perp)$. This is not the case, the canonical extension of a complete Boolean algebra is not usually an isomorphism.
}
\end{rem}
\noindent We now consider Goldblatt's \ts{of} in the context of a monadic \ts{ol}.

\begin{prop}\label{stuff2}
Let $L$ be a monadic \ts{ol} and define a binary relation $R$ on its Goldblatt~\ts{of}  by $x\,R\,y$ iff $\Ex[x]\subseteq y$. Then $X=(\mc{F}(L),\perp,R)$ is a monadic \ts{of} and $h:L\to\mc{B}(X)$ is the canonical extension of $L$.    
\end{prop}

\begin{proof}
To show that $X$ is a monadic \ts{of} we must first show that $R$ is reflexive and transitive. For reflexivity, let $a\in\exists[x]$ so that $a=\exists b$ for some $b\in x$. Since $b\leq\exists b$ and $x$ is upward closed, we have $\exists b=a\in x$ and therefore $xRx$.  For transitivity,   assume $xRy$ and $yRz$ so that $\exists[x]\subseteq y$ and $\exists[y]\subseteq z$. Let $a\in\exists[x]$ so that $a=\exists b$ for some $b\in x$. Then $\exists b\in y$ and hence $\exists\exists b\in\exists[y]$ which implies $\exists\exists b\in z$. Then $\exists\exists b=\exists b$ and $a=\exists b$ so $a\in z$. Therefore $\exists[x]\subseteq z$ and hence $xRz$ so we conclude $R$ is transitive.

We now show that $R$ satisfies $R[R[\{x\}]^{\perp}]\subseteq R[\{x\}]^{\perp}$ for all $x\in \mathcal{F}(L)$. Let $z$ be the filter generated by $\Ex[x]$ and note that $z$ is the smallest filter belonging to $R[\{x\}]$. Thus, $R[\{x\}]^\perp = \{z\}^\perp$. 
If $y\in \{z\}^\perp$, then there is $b\in z$ with $b'\in y$. Since $z$ is the filter generated by $\Ex[x]$, there are $a_1,\ldots,a_n\in x$ with $\Ex a_1\wedge\cdots\wedge \Ex a_n\leq b$. Set $a=a_1\wedge\cdots\wedge a_n$. Then $a\in x$ and we have $\Ex a\leq \Ex a_1\wedge\cdots\wedge \Ex a_n\leq b$. Thus $b'\leq (\Ex a)'$ so $(\Ex a)'$ belongs to $y$. It follows that 
\[R[\{x\}]^\perp = \{z\}^\perp = \{y:(\Ex a)'\in y\mbox{ for some }a\in x\}.\]
Suppose $y\in 
R[\{x\}]^\perp$ and $y\,R\,w$. Then there is $a\in x$ with $(\Ex a)'\in y$ and $\Ex[y]\subseteq w$. Since~$L$ is a monadic \ts{ol}, we have $(\Ex a)'=\Ex(\Ex a)'\in w$, giving $w\in R[\{x\}]^\perp$ as required. So $X$ is a monadic \ts{of}. 

For $a\in L$ we have that $h(a)=\{x:a\in x\}$. Then $R[h(a)]=\{x:\Ex a\in x\}$. It follows that $R[h(a)]$ is bi-orthogonally closed, so $h(\Ex a) = R[h(a)] = R[h(a)]^{\perp\perp} = \Ex h(a)$. So, using $\mc{B}$ for $\mc{B}(X)$, we have that $h:L\to\mc{B}$ is a monadic \ts{ol}-embedding. We further know that when restricted to the \ts{ol} reduct, this is the canonical extension. It remains to show that the quantifier that we denote $\Ex_R$ of $\mc{B}$ is the canonical extension $\Ex^\sigma$ of quantifier of $L$, as described above. 

Let $\mc{K}$ be the set of closed elements of $\mc{B}$, that is, those that are a meet of elements in the image of $h$, and for each filter $x$ of $L$ let $K_x=\{y:x\subseteq y\}$. Let $S\subseteq L$. If the filter $x$ generated by $S$ is proper, we have $\bigwedge h[S]=K_x$, and $\bigwedge h[S]=\emptyset$ if $x$ is improper. Thus, each non-empty $K\in\mc{K}$ is of the form $K_x$ for some proper filter $x$ of $L$, and it is easily seen that this $x$ is unique. 

The definition of $\Ex^\sigma$ gives 
\[\Ex^\sigma K_x = \bigwedge \{h(\Ex a):K_x\subseteq h(a)\}.\]
Since $K_x\subseteq h(a)$ iff $a\in x$, $\Ex^\sigma K_x$ is equal to $\bigcap\{h(\Ex a):a\in x\}$, which in turn is equal to $\{y:\Ex[x]\subseteq y\}$, and hence is given by $R[K_x]$. But this set is bi-orthogonally closed since it is the intersection of bi-orthogonally closed sets, so $\Ex^\sigma K_x=R[K_x]^{\perp\perp}=\Ex_R K_x$. 

Suppose that $A$ is any element of $\mc{B}$. Then, the definition of $\Ex^\sigma$ and the result just established for closed elements $K_x$ gives
\[\Ex^\sigma A = \bigvee \{\Ex^\sigma K_x:K_x\leq A\}=\bigvee\{\Ex_RK_x:K_x\leq A\}.\]
Since $\Ex_R$ is order preserving, $\Ex^\sigma A\subseteq \Ex_RA$. To see equality, suppose $y\in R[A]$. Then $x\,R\,y$ for some $x\in A$. Since $A$ is bi-orthogonally closed, it is an upset in the poset of proper filters, so $K_x\subseteq A$. But $y\in R[K_x]=\Ex_R K_x$. Thus $R[A]\subseteq\bigcup\{\Ex_RK_x:K_x\leq A\}$. Using the established fact that $\Ex^\sigma K_x=\Ex_RK_x$ and taking the bi-orthogonal closure of both sides gives $\Ex_RA\leq \bigvee\{\Ex^\sigma K_x:K_x\leq A\}=\Ex^\sigma A$. 
\end{proof}

To conclude this section, we recall that in the classical setting, one associates to a monadic algebra $(B,\Ex)$ a set $X$ with an equivalence relation $S$ on $X$. We show that this path can also be taken with a monadic \ts{ol}. 

\begin{defn}
For $L$ a monadic \ts{ol}, define a relation $S$ on its set $\mf{F}(L)$ of proper filters by $x\,S\,y \mbox{ iff }\Ex[x]=\Ex[y]$.    
\end{defn}

Clearly $S$ is an equivalence relation. To see it's further properties, we use an auxilliary relation $\uparrow$ on $\mf{F}(L)$ where $x\uparrow y$ iff $x\subseteq y$. We then write ${\uparrow}S$ for the composite ${\uparrow}\circ S $. So $x\,{\uparrow}S\,z$ iff there is $y$ with $\Ex[x]=\Ex[y]$ and $y\subseteq z$. 

\begin{lem}
$R={\uparrow}S$. 
\end{lem}

\begin{proof}
Given $x$, let $\widehat{x}$ be the filter generated by $\Ex[x]$ and note that since $x$ is down-directed, $\widehat{x}$ is the upset generated by $\Ex[x]$. Then $x\, R\, z$ iff $\Ex[x]\subseteq z$ iff $\widehat{x}\subseteq z$. Since $\Ex[x]=\Ex[\widehat{x}]$ it follows that $x\,R\,z$ implies $x\, {\uparrow}S\, z$. Conversely, if $x\,{\uparrow}S\,z$, there is $y$ with $\Ex[x]=\Ex[y]$ and $y\subseteq z$. Since $\Ex[x]\subseteq y$, then $\widehat{x}\subseteq y$, hence $\widehat{x}\subseteq z$ and so $x\, R\, z$. 
\end{proof}

\begin{prop}
Let $L$ be a monadic \ts{ol}. Then $Y=(\mc{F}(L),\perp,S)$ is a monadic \ts{of}, $S$ is an equivalence relation, and $h:L\to \mc{B}(Y)$ is the canonical extension of $L$.    
\end{prop}

\begin{proof}
Since $x\leq y$ and $x\perp z$ implies $y\perp z$, for any $A\subseteq Y$ we have $A^\perp = ({\uparrow}A)^\perp$. Thus $S[\{x\}]^\perp = ({\uparrow}S[\{x\}])^\perp = R[\{x\}]^\perp$. Then, since $S[A]\subseteq R[A]$ for any subset $A$, we have $$S[S[\{x\}]^\perp]\subseteq R[R[\{x\}]^\perp]\subseteq R[\{x\}]^\perp = S[\{x\}]^\perp.$$ 
Here, the second containment uses the fact that $X=(\mf{F}(L),\perp,R)$ is a monadic \ts{of}. This inequality shows that $Y$ is a monadic \ts{of}, and as noted, $S$ is an equivalence relation. Since the mapping $h$ does not depend on the choice of $R$ or $S$, it is an \ts{ol} embedding into $\mc{B}(Y)$ that provides a canonical extension of $L$ when considered as a \ts{ol}. To show that it is a canonical extension of $L$ as a monadic \ts{ol}, we show that the quantifiers $\Ex_R$ and $\Ex_S$ on $\mc{B}(\mf{F}(L),\perp)$ are equal. But \[\Ex_R A = R[A]^{\perp \perp} = ({\uparrow}S[A])^{\perp \perp}= S[A]^{\perp \perp}=\Ex_SA.\] This completes the proof.  
\end{proof}

Our derivation could have been done throughout starting with the relation $S$, but it would have been a bit more complicated.

\section{Orthospaces}

Bimb\'o \cite{bimbo} placed Goldblatt's work \cite{goldblatt} on the Stone space of a \ts{ol} in a categorical setting in an attempt to create a duality between the category {\sf OL} of \ts{ol}s and their homomorphisms and what she called the category of orthospaces. 

The focus of this section is to show that what is obtained in \cite{bimbo} is a dual adjunction that gives rise to a dual equivalence when the definition of orthospaces is appropriately extended.  
\begin{defn}
An orthospace (abbrev.: \ts{os}) $(X,\perp,\leq,\tau)$ consists of an \ts{of} with a partial ordering $\leq$ and a compact topology $\tau$ that satisfies 
\begin{enumerate}
\item if $x\nleq y$ then there is $U\in\mc{C}(X)$ with $x\in U$ and $y\not\in U$,
\item if $x\perp z$ and $x\leq y$, then $y\perp z$, 
\item if $U\in\mc{C}(X)$, then $U^\perp\in\mc{C}(X)$, 
\item if $x\perp y$, then there is $U\in\mc{C}(X)$ with $x\in U$ and $y\in U^\perp$. 
\end{enumerate}
Here $\mc{C}(X)$ is the set of clopen, bi-orthogonally closed subsets of $X$.
\end{defn}
Condition (1) guarantees that every \ts{os} $X$ is totally-order disconnected and therefore totally disconnected. Since $X$ is compact by definition, $X$ is a Stone space.     
\begin{lem}\label{fig}
In any \ts{os} we have $x\leq y$ iff $y\in\{x\}^{\perp\perp}$.  
\end{lem}

\begin{proof}
If $x\leq y$ then $y\in\{x\}^{\perp\perp}$ since (2) implies that each bi-orthogonally closed set is an upset and we always have $x\in\{x\}^{\perp\perp}$. Conversely, if $x\not\leq y$, then by (1) there is $U\in\mc{C}(X)$ with $x\in U$ and $y\not\in U$. Since $U$ is bi-orthogonally closed and $x\in U$, we have $\{x\}^{\perp\perp}\subseteq U$, hence $y\not\in\{x\}^{\perp\perp}$.  
\end{proof}

Using this lemma, one can formulate an equivalent definition of an \ts{os} that does not involve an ordering. Call $(X,\perp,\tau)$ an \ts{os}$^\prime$ if $(X,\perp)$ is an \ts{of} with a compact  topology~$\tau$ that satisfies (3) and (4) and additionally has $\mc{C}(X)$ separate points, that is, it satisfies the following: (1$^\prime$) if $x\neq y$ there is $U\in\mc{C}(X)$ with  $U\cap\{x,y\}$ containing one element. Clearly if $(X,\perp,\leq,\tau)$ is an \ts{os}, then $(X,\perp,\tau)$ is an \ts{os}$^\prime$ since (1) implies (1$^\prime$).

\begin{prop}
If $(X,\perp,\tau)$ is an \ts{os}$^\prime$, then setting $x\leq y$ iff $y\in\{x\}^{\perp\perp}$, we have that $(X,\perp,\leq,\tau)$ is an \ts{os}.    
\end{prop}

\begin{proof}
It is simple to see that $\leq$ is reflexive and transitive. If $x\neq y$, then by (1$^\prime$)
there is $U$ that separates them, say $x\in U$ and $y\not\in U$. It follows that $\{x\}^{\perp\perp}\subseteq U$, hence $y\not\in \{x\}^{\perp\perp}$. So $\leq$ is anti-symmetric, hence a partial order. If $x\nleq y$, then $y\not\in\{x\}^{\perp\perp}$. So there is $z$ with $x\perp z$ and $y\not\perp z$. By (4) there is $U\in\mc{C}(X)$ with $x\in U$ and $z\in U^\perp$. We cannot have $y\in U$ since that would give $y\perp z$, hence $y\not\in U$. This shows that (1) holds for our derived relation $\leq$. Suppose $x\perp z$ and $x\leq y$. Then $y\in\{x\}^{\perp\perp}$. But $z\in\{x\}^\perp$, so $y\perp z$, giving (2).  
\end{proof}

If we begin with an \ts{os}$^\prime$ $(X,\perp,\tau)$, then form an \ts{os}
$(X,\perp,\leq,\tau)$ as above, then build from it an \ts{os}$^\prime$, we obviously return to the original since we have merely created and then discarded an auxiliary relation $\leq$. Suppose we start with an \ts{os} $(X,\perp,\leq,\tau)$, and then use the \ts{os}$^\prime$ to form a partial ordering. By Lemma~\ref{fig}, we return to our original \ts{os}. Thus, the notions of \ts{os} and \ts{os}$^\prime$ are equivalent. We follow Bimb\'o's terminology to make it easy to match with her paper, although the notion of an \ts{os}$^\prime$ seems simpler.

\begin{defn}\label{of-morphism}
Let $(P,\perp,\leq,\tau)$ and $(X,\perp,\leq,\tau)$ be \ts{os}'s. A function $\phi:P\to X$ is an \ts{os} \emph{morphism} if $\phi$ is continuous and satisfies 
\begin{enumerate}
\item if $\phi(p) \perp \phi(q)$ then $p\perp q$,
\item  if $x\not\perp \phi(p)$ then there exists $q$ with $q\not\perp p$ and $\phi(q)\in \{x\}^{\perp\perp}$.  
\end{enumerate}
Let $\sf{OS}$ be the category of \ts{os}'s and their morphisms. 
\end{defn}

For an \ts{ol} $L$, Goldblatt \cite{goldblatt} considered the topology $\tau$ on $\mc{F}(L)$ having as a sub-basis all sets $h(a)$, and their set-theoretic complements, for $a\in L$. He showed that this is a Stone topology, that the clopen bi-orthogonally closed sets of $\mc{F}(L)$ form an \ts{ol}, and that $h$ is an isomorphism from $L$ to the \ts{ol} of clopen bi-orthogonally closed sets of $\mc{F}(L)$. Bimb\'o showed \cite[Lemma~3.4]{bimbo} that $(\mathcal{F}(L),\subseteq,\perp,\tau)$ is an \ts{os}, and that for any \ts{os} $X$, its clopen bi-orthogonally closed sets $\mc{C}(X)$ form an \ts{ol} \cite[Lemma~3.3]{bimbo}. So for any \ts{ol} $L$ we have an \ts{os} $\mc{F}(L)$, and for any \ts{os} $X$ we have an \ts{ol} $\mc{C}(X)$. 

\begin{prop}
These assignments on objects extend to contravariant functors $\mc{F}:\sf{OL}\to\sf{OS}$ and $\mc{C}:\sf{OS}\to\sf{OL}$. For $f:L\to M$ an \ts{ol}-homomorphism, $\mc{F}(f):\mc{F}(M)\to\mc{F}(L)$ is given by $\mc{F}(f)=f^{-1}[\,\cdot\,]$ and for $\phi:P\to X$ an \ts{os}-morphism, $\mc{C}(\phi):\mc{C}(X)\to\mc{C}(P)$ is given by $\mc{C}(\phi)=\phi^{-1}[\,\cdot\,]$. 
\end{prop}
\begin{proof}
See \cite[Lemmas~3.9 and~3.10]{bimbo} for the proofs. 
\end{proof}

It was shown in \cite{bimbo} that there are a pair of natural transformations $h:1_{\sf{OL}}\to\mc{CF}$ and $g:1_{\sf{OS}}\to\mc{FC}$ where, for an \ts{ol} $L$ and \ts{os} $X$, the components of the natural transformations are given by:
\begin{align*}
h_L:L\to\mc{CF}(L) & \mbox{ is given by } h_L(a) =\{x:a\in x\}, \\   
g_X:X\to\mc{FC}(X) & \mbox{ is given by }
g_X(x)=\{U:x\in U\}.
\end{align*}
The naturality of $g$ and $h$ is given in 
\cite[Thm.~3.11]{bimbo}. The proof that $h_L$ is an \ts{ol}-isomorphism is given in \cite{goldblatt}.  
Thus $h:1_{\sf{OL}}\to\mc{CF}$ is a natural isomorphism. In \cite[Thm.~3.6]{bimbo}, it is claimed that $g_X$ is an isomorphism, hence that $\mc{F}$ and $\mc{C}$ provide a dual equivalence between {\sf OL} and {\sf OS}.
The proof that $g_X$ is one-one is correct, but it need not be onto. The issue with $g_X$ being onto was first pointed out in \cite{dmitrieva}. Below, we provide an example to show that $g_X$ need not be onto. 

\begin{ex}
For $X=\{x,y\}$, let $\perp$ be the relation $\not=$ of inequality, $\leq$ be the relation $=$ of equality, and $\tau$ be the discrete topology. Then $\tau$ is a Stone topology on $X$, $\perp$ is irreflexive and symmetric, hence an orthogonality relation on $X$, and $\leq$ is a partial ordering. Moreover, every subset of $X$ is both clopen and bi-orthogonally closed, so $\mc{C}(X)$ is the powerset of $X$ and $U^\perp$ is the set-theoretic complement of $U$ for each $U\subseteq X$. It is a simple matter to verify that $(X,\perp,\leq,\tau)$ is an \ts{os}. Then $\mc{C}(X)$ is a 4-element Boolean algebra, since it is the powerset of the 2-element set $X$. But a 4-element Boolean algebra has 3 proper filters. So $\mc{FC}(X)$ is a 3-element \ts{os}, thus cannot be isomorphic to $X$.
\end{ex}

\begin{defn}
An \ts{os} $X$ is \emph{ortho-sober} if each proper filter in the ortholattice $\mathcal{C}(X)$ is equal to $\{U\in\mc{C}(X):x\in U\}$ for some $x\in X$. 
\end{defn}

Ortho-sober orthospaces were introduced by Dmitrieva \cite{dmitrieva}, and later considered by McDonald and Yamamoto \cite{mcdonald}. The point is that for the full sub-category {\sf OSOS} of {\sf OS}
consisting of ortho-sober \mbox{orthospaces}, $\mc{F}$ maps {\sf OL} into {\sf OSOS}, and then $\mc{F}$ and $\mc{C}$ provide a dual equivalence between {\sf OL} and {\sf OSOS}. We use the remainder of this section to formulate what exists in the approach of \cite{bimbo} without the introduction of the ortho-sober condition. 

\begin{thm}
The functors $\mc{F}\dashv\mc{C}$ provide an adjunction between {\sf OL} and ${\sf OS}^{op}$. 
\end{thm}

\begin{proof}
For $L$ and \ts{ol} and $X$ an \ts{os} we define mappings $(\,\cdot\,)^-$ and $(\,\cdot\,)^+$
\begin{center}
\begin{tikzpicture}
\node at (0,0) {$\ts{Hom}_{\sf OL}(L,\mc{C}(X))$};  
\node at (5,0) {$\ts{Hom}_{\sf OS}(X,\mc{F}(L))$}; 
\draw[->] (1.6,0.1)--(3.4,0.1);
\draw[->] (3.4,-0.1)--(1.6,-0.1);
\node at (2.6,.5) {$(\,\cdot\,)^-$};
\node at (2.6,-.5) {$(\,\cdot\,)^+$};
\end{tikzpicture}    
\end{center}
by setting for $f:L\to\mc{C}(X)$ and $\phi:X\to\mc{F}(L)$
\begin{align*}
f^-(x) &= \{a:x\in f(a)\},\\
\phi^+(a) &= \{x:a\in\phi(x)\}.
\end{align*}
Note that $f^-:X\to\mc{F}(L)$ is the composite of $g_X:X\to\mc{FC}(X)$ and $\mc{F}(f):\mc{FC}(X)\to\mc{F}(L)$. Indeed, we have \[\mc{F}(f)\circ g_X (x) = \mc{F}(f)(\{U:x\in U\}) = \{a:x\in f(a)\}.\] 
Note also that $\phi^+:L\to\mc{C}(X)$ is the composite of $h_L:L\to\mc{CF}(L)$ and $\mc{C}(\phi):\mc{CF}(L)\to\mc{C}(X)$. Indeed, we have \[\mc{C}(\phi) \circ h_L(a) = \mc{C}(\phi) (\{y:a\in y\}) = \{x:a\in\phi(x)\}.\] Thus, the maps $(\,\cdot\,)^-$ and $(\,\cdot\,)^+$ between hom-sets are well-defined. Also, we have 
\begin{align*}
f^{-+}(a)&=\{x:a\in f^-(x)\}=\{x:x\in f(a)\}=f(a), \\
\phi^{+-}(x) &= \{a:x\in \phi^+(a)\} = \{a:a\in\phi(x)\} =\phi(x).
\end{align*}
Thus, for each $L,X$ the maps $(\,\cdot\,)^-$ and $(\,\cdot\,)^+$ are mutually inverse bijections between hom-sets. We require naturality. For naturality in one coordinate, we must show that if $\psi:X'\to X$ and $\alpha:X\to\mc{F}(L)$, then $(\alpha\circ\psi)^+=\mc{C}(\psi)\circ\alpha^+$. For naturality in the other coordinate, we must show that if $f:L'\to L$ and $\alpha:X\to\mc{F}(L)$, then $(\mc{F}(f)\circ\alpha)^+=\alpha^+\circ f$. For the former, we have 
\begin{align*}
    (\mc{C}(\psi)\circ\alpha^+)(a)&=\mc{C}(\psi)(\{x: a\in\alpha(x)\})\\&= \{x':\psi(x')\in\{x:a\in\alpha(x)\}\}
    \\&=\{x':a\in\alpha\psi( x')\} \\&= (\alpha\circ\psi)^+(a),
\end{align*}
and for the latter, we have
\begin{align*}
    (\alpha^+\circ f)(b) &= 
\{x:f(b)\in\alpha(x)\} \\&=
\{x:b\in (\mc{F}(f)\circ\alpha)(x)\}
\\&= (\mc{F}(f)\circ \alpha)^+(b).
\end{align*}
This completes the proof. 
\end{proof}

\section{Monadic orthospaces}

In this section, we extend the results in the previous section to the setting of monadic \ts{ol}s. 

\begin{defn}\label{monadic orthospace}
A tuple $(X;\perp,\leq,R,\tau)$ is a monadic orthospace if $(X,\perp,\leq,\tau)$ is an \ts{os}, $(X,\perp,R)$ is a monadic \ts{of}, and for each $U\in\mc{C}(X)$ we have $R[U]\in\mc{C}(X)$ .
\end{defn} 

For a monadic \ts{of} $X$, its bi-orthogonally closed sets $\mc{B}(X)$ form a monadic \ts{ol} under the quantifier $\Ex A = R[A]^{\perp\perp}$. It is clear from the definition of a monadic \ts{os} that its clopen bi-orthogonally closed sets $\mc{C}(X)$ form a subalgebra of $\mc{B}(X)$, hence form a monadic \ts{ol}. 

\begin{defn} 
Let $L$ be a monadic \ts{ol}, and equip its Goldblatt \ts{os} $(\mc{F}(L),\perp,\leq,\tau)$ with the relation $xRy$ iff $\Ex[x]\subseteq y$ of its Goldblatt \ts{of}. Call this the monadic Goldblatt \ts{os} and denote it $\mc{F}L$.
\end{defn}

For a monadic \ts{ol} $L$, we have that $\mc{F}(L)$ is indeed a monadic \ts{os}. It is clearly an \ts{os} and also a monadic \ts{of}. It remains only to show that if $U$ is a clopen and bi-orthogonally closed set of $\mc{F}(L)$, then so is $R[U]$. By Goldblatt's result, $U=h(a)$ for some $a\in L$. In the proof of Proposition~\ref{stuff}, we saw that $R[h(a)]=h(\Ex a)$, so $R[U]$ is clopen and bi-orthogonally closed, hence $\mc{F}(L)$ is a monadic \ts{os}. 

\begin{prop}
For $L$ a monadic \ts{ol}, the map $ h_L:L\to\mc{CF}(L)$ given by $h_L(a)=\{x : a\in x\}$ is a monadic \ts{ol} isomorphism. 
\end{prop}

\begin{proof}
We know that $h_L$ is an \ts{ol} isomorphism. By the discussion above, for $a\in L$ we have $h_L(\Ex a)=R[h(a)]=R[h(a)]^{\perp\perp}=\Ex h_L(a)$ and thus $h_L$ is a homomorphism for $\exists$.    
\end{proof}

\begin{defn}
For monadic \ts{os}'s $X$ and $Y$, a map $\phi:X\to Y$ is a monadic \ts{os} morphism if it is an \ts{os} morphism and $R[\phi^{-1}[U]]=\phi^{-1}[R[U]]$ for each $U\in\mc{C}(Y)$.     
\end{defn}

\begin{prop}
For $X$ a monadic \ts{os}, the map $g_X:X\to \mc{FC}(X)$ given by $g_X(x)=\{U:x\in U\}$ is a monadic \ts{os} embedding.    
\end{prop}

\begin{proof}
To aid readability, we write $g$ for $g_X$. It is known that $g$ is an \ts{os} embedding. The remaining condition for $g$ to be a monadic \ts{os} morphism involves several levels. To assist with this, we use the following conventions. Elements of $X$ are written $x,y$ and $R$ is the additional relation on $X$. Elements of $\mc{FC}(X)$ are filters of $\mc{C}(X)$ and are written as $F$. The relation of $\mc{FC}(X)$ is denoted $S$. We must show that for $\mc{V}\in\mc{CFC}(X)$
\[ R[g^{-1}[\mc{V}]]=g^{-1}[S[\mc{V}]].\]

\noindent Note that by Goldblatt's result, there is some $U_0\in\mc{C}(X)$ with \[\mc{V}=\{F\in\mc{FC}(X):U_0\in F\}.\]    

Observe that $x\in g^{-1}[\mc{V}]$ iff $g(x)\in\mc{V}$, which occurs iff $U_0\in g(x)$, hence, iff $x\in U_0$. Thus 
\[R[g^{-1}[\mc{V}]] \,\, = \,\, R[U_0].\] 
For $y\in X$
we have $y\in g^{-1}[S[\mc{V}]]$ iff $g(y)\in S[\mc{V}]$. This occurs iff there is some $F\in\mc{V}$ with $F\, S\,g(y)$, hence some filter $F$ with $U_0\in F$ and $\Ex_{\mc{C}(X)}[F]\subseteq g(y)$. But ${\uparrow}U_0$, the principle filter of $\mc{C}(X)$ generated by $U_0$, is the smallest filter containing $U_0$. So these conditions occur iff $\Ex_{\mc{C}(X)}U_0\in g(y)$, and this occurs iff $y\in\Ex_{\mc{C}(X)}U_0$. But we have seen that $\Ex_{\mc{C}(X)}U_0 = R[U_0]$. Thus 
\[ g^{-1}[S[\mc{V}]]\,\, = \,\, R[U_0].\]
This establishes the result. 
\end{proof}

%

\begin{lem}
The composite of monadic \ts{os} morphisms is a monadic \ts{os} morphism.     
\end{lem}

\begin{proof}
Suppose $X,Y,Z$ are monadic \ts{os}'s with associated relations $R,S,T$, respectively. Let 
$\phi:X\to Y$ and $\psi:Y\to Z$ be monadic \ts{os} morphisms. Then for $W\in\mc{C}(Z)$ we have that $\psi^{-1}[W]\in\mc{C}(Y)$ since $\psi$ is in particular an \ts{os} morphism, and then $S[\psi^{-1}[W]]$ also belongs to $\mc{C}(Y)$ since $Y$ is a monadic \ts{os}. Then $\phi^{-1}\psi^{-1}[T[W]]= \phi^{-1}[S[\psi^{-1}[W]]] = R[\phi^{-1}\psi^{-1}[W]]$.      
\end{proof}

Let {\sf mOL} be the category of monadic \ts{ol}s (usually {\sf MOL} denotes modular \ts{ol}s) and {\sf mOS} be the category of monadic \ts{os}s. For $L$ a monadic \ts{ol}, its Goldblatt frame $\mc{F}(L)$ is a monadic \ts{os}, and for a monadic \ts{os} $X$ its clopen bi-orthogonally closed subsets $\mc{C}(X)$ form a monadic \ts{ol}. 

\begin{lem}
For $f:L\to M$ an {\sf mOL} morphism and $\phi:X\to Y$ an {\sf mOS} morphism, $f^{-1}:\mc{F}(M)\to\mc{F}(L)$ is an {\sf mOS} morphism, and $\phi^{-1}:\mc{C}(Y)\to\mc{C}(X)$ is an {\sf mOL} morphism.    
\end{lem}

\begin{proof}
We first show that $\psi=f^{-1}$ is an {\sf mOS} morphism. Since we already know it is an \ts{os} morphism, it remains to show that for $U\in\mc{CF}(L)$, we have $R[\psi^{-1}[U]]=\psi^{-1}[R[U]]$. By Goldblatt's result, $U=h_L(a)$ for some $a\in L$. Earlier we showed that $R[h_L(a)]=h_L(\Ex a)$ and $\psi^{-1}[h_L(a)]=h_M(f(a))$. Thus, we have
\begin{align*}
    R[\psi^{-1}[U]]&= R[h_M(f(a))]=h_M(\Ex f(a))=h_M(f(\Ex a))\\&= \psi^{-1}[h_L(\Ex a)] = \psi^{-1}[R[U]].
\end{align*}

Since $\phi:X\to Y$ is an \ts{os} morphism, $\phi^{-1}$ is an \ts{ol} homomorphism. We must show that for $U\in\mc{C}(Y)$, that $\phi^{-1}[\Ex U] = \Ex\phi^{-1}[U]$. But $U\in\mc{C}(Y)$ implies that $R[U]\in\mc{C}(Y)$, and so $\Ex U = R[U]^{\perp\perp}=R[U]$. Therefore, since $\phi$ is a {\sf mOS} morphism, we have \[\phi^{-1}[\Ex U]=\phi^{-1}[R[U]]=R[\phi^{-1}[U]] = \Ex\phi^{-1}[U],\] which completes the proof. 
\end{proof}

We then have that $\mc{F}$ and $\mc{C}$ are contravariant functors between {\sf mOL} and {\sf mOS}. Recall, that for an \ts{ol} $L$ and \ts{os} $X$, we earlier produced mutually inverse bijections between homsets where $f^-=\mc{F}(f)\circ g_X$ and $\phi^+=\mc{C}(\phi)\circ h_L$. 
\begin{center}
\begin{tikzpicture}
\node at (0,0) {$\ts{Hom}_{\sf OL}(L,\mc{C}(X))$};  
\node at (5,0) {$\ts{Hom}_{\sf OS}(X,\mc{F}(L))$}; 
\draw[->] (1.6,0.1)--(3.4,0.1);
\draw[->] (3.4,-0.1)--(1.6,-0.1);
\node at (2.6,.5) {$(\,\cdot\,)^-$};
\node at (2.6,-.5) {$(\,\cdot\,)^+$};
\end{tikzpicture}    
\end{center}
If $L$ is an {\sf mOL}, $X$ an {\sf mOS}, $f$ is an {\sf mOL} morphism, and $\phi$ an {\sf mOS} morphism, then since $g_X$ is an {\sf mOS} morphism, and $h_L$ is an {\sf mOL} morphism, it follows that $f^-$ is an {\sf mOS} morphism and $\phi^+$ is an {\sf mOL} morphism. Thus, we have mutually inverse bijecitons 
\begin{center}
\begin{tikzpicture}
\node at (0,0) {$\ts{Hom}_{\sf mOL}(L,\mc{C}(X))$};  
\node at (5,0) {$\ts{Hom}_{\sf mOS}(X,\mc{F}(L))$}; 
\draw[->] (1.6,0.1)--(3.4,0.1);
\draw[->] (3.4,-0.1)--(1.6,-0.1);
\node at (2.6,.5) {$(\,\cdot\,)^-$};
\node at (2.6,-.5) {$(\,\cdot\,)^+$};
\end{tikzpicture}    
\end{center}
\noindent The naturality of these in each coordinate is given by the naturality in the previous setting. This yields the following. 

\begin{thm}
$\mc{F}\dashv \mc{C}$ is an adjunction between {\sf mOL} and {\sf mOS}$^{op}$.    
\end{thm}

Recall that an \ts{os} $X$ is ortho-sober if each proper filter of $\mc{C}(X)$ is equal to $\{U:x\in U\}$ for some $x\in X$. This is equivalent to having $g_X:X\to\mc{FC}(X)$ be an isomorphism. The dual adjunction between {\sf OL} and {\sf OS} restricts to a dual equivalence between {\sf OL} and the full subcategory of ortho-sober \ts{os}s. 

\begin{cor}
There is a dual equivalence between {\sf mOL} and the full subcategory of {\sf mOS} consisting of ortho-sober monadic \ts{os}s. 
\end{cor}
\section*{Declarations}
\subsection*{Data availability} Data sharing not applicable to this article as datasets were neither generated
nor analysed.
\subsection*{Compliance with ethical standards} The authors declare that they have no conflict of interest.


\begin{thebibliography}{99}
\setlength{\parskip}{0em}
    \bibitem{bimbo} Bimb\'o, K.: Functorial duality for ortholattices and De Morgan lattices. Log. Univ. \textbf{1}, 311--333 (2007)
    
   \bibitem{BLT} Birkhoff, G.: Lattice theory, Third edition. American Mathematical Society Colloquium Publications, Vol. XXV, American Mathematical Society, Providence (1967)
   
     \bibitem{dmitrieva}  Dmitrieva, A.: Positive modal logic beyond distributivity:
duality, preservation and completeness. MSc. Thesis, University of Amsterdam, ILLC Publications (2021)
\bibitem{harding2} Gehrke, M., Harding, J.: Bounded lattice expansions. Journal of Algebra. \textbf{238}, no. 1 (2001)
     \bibitem{Gehrke} Gehrke, M., Harding, J., Venema, Y.: MacNeille completions and canonical extensions. Trans. Amer. Math. Soc. \textbf{358}, 573--590 (2005)
     \bibitem{goldblatt} Goldblatt, R.: The Stone space of an ortholattice. Bull. Lond. Math. Soc. \textbf{7}, 45--48 (1975)
     \bibitem{halmos} Halmos, P.: Algebraic Logic. Chelsea Publishing Company, New York (1962)
    \bibitem{harding} Harding, J.: Quantum monadic algebras. J. Phys. A. \textbf{55} no. 39 (2023) 
\bibitem{tarski1}
 Henkin, L., Monk, J.D., Tarski, A.: Cylindric algebras. Part I, volume 64 of Studies
in Logic and the Foundations of Mathematics. North-Holland Publishing Co., Amsterdam (1985)
\bibitem{tarski2} Henkin, L., Monk J.D., and Tarski, A.: Cylindric algebras. Part II, volume 115 of Studies
in Logic and the Foundations of Mathematics. North-Holland Publishing Co., Amsterdam (1985)
    
\bibitem{jonsson} J\'onsson, B, Tarski, A.: Boolean algebras with operators. I. Amer. J. Math. \textbf{73}, 891--939 (1951)

\bibitem{jonssonII} J\'onsson, B, Tarski, A.: Boolean algebras with operators. II. Amer. J. Math. \textbf{74}, 127--162 (1952)

\bibitem{janowitz} Janowitz, M.: Quantifiers and orthomodular lattices. Pacific J. Math. \textbf{13}, 1241--1249 (1963)

\bibitem{maclaren} MacLaren, M.D.: Atomic orthocomplemented lattices.  Pacific J. Math. \textbf{14}, 597--612 (1964)

\bibitem{mcdonald} McDonald, J., Yamamoto, K.: Choice-free duality for orthocomplemented lattices by means of spectral spaces. Algebra Universalis. \textbf{83}, no. 37 (2022)






\end{thebibliography}
\end{document}